\documentclass[oneside,english]{amsart}
\usepackage[pdfencoding=auto,psdextra]{hyperref}
\usepackage[T1]{fontenc}
\usepackage[latin9]{inputenc}
\usepackage{geometry}
\usepackage{amstext}
\usepackage{amsthm}
\usepackage{amssymb}
\usepackage{graphicx}
\usepackage[outdir=./]{epstopdf}

\hypersetup{colorlinks,
	citecolor=red,
	linkcolor=blue,}
\usepackage{color}
\makeatletter
\numberwithin{equation}{section}
\numberwithin{figure}{section}
\theoremstyle{plain}
\newtheorem{thm}{\protect\theoremname}
  \theoremstyle{plain}
  \newtheorem{cor}[thm]{\protect\corollaryname}
  \theoremstyle{definition}
  \newtheorem{defn}[thm]{\protect\definitionname}
  \theoremstyle{plain}
  \newtheorem{lem}[thm]{\protect\lemmaname}
  \theoremstyle{plain}

\makeatother

\usepackage{babel}
  \providecommand{\corollaryname}{Corollary}

  \providecommand{\definitionname}{Definition}
  \providecommand{\lemmaname}{Lemma}
  \providecommand{\propositionname}{Proposition}
\providecommand{\theoremname}{Theorem}
  \providecommand{\remarkname}{Remark}

\newcommand{\R}{\mathbb{R}}

\newcommand{\cbinom}[2]{\begin{Bmatrix} #1 \\ #2 \end{Bmatrix}}

\DeclareMathOperator{\vol}{vol}

\begin{document}

\title{A Continuous Analogue of Lattice Path Enumeration: Part II}

\author{Tanay Wakhare$^{\dag}$}
\address{$^{\dag}$~University of Maryland, College Park, MD 20742, USA}
\email{twakhare@gmail.com}
\author{Christophe Vignat$^{\ddag}$}
\address{$^\ddag$~Tulane University, New Orleans, LA 70118, USA}
\email{cvignat@tulane.edu} 

\begin{abstract}
Following the work of Cano and D\'iaz, we study continuous binomial coefficients and Catalan numbers. We explore their analytic properties, including integral identities and generalizations of discrete convolutions. We also conduct an in-depth analysis of a continuous analogue of the binomial distribution, including a stochastic representation as a Goldstein-Kac process.
\end{abstract}

\maketitle

\section{Introduction}

In two recent papers \cite{Diaz1, Diaz2}, Leonardo Cano and Rafael D\'iaz introduced continuous analogues of the binomial coefficients and Catalan numbers. They did this by introducing a general procedure for studying continuous lattice paths, then measuring the volume of a moduli space associated to these continuous paths. By recognizing the binomial coefficients and Catalan numbers as counting certain types of lattice paths, their continuous analogues are defined as the volumes of associated moduli spaces.

In Part I of this work \cite{Wakhare}, we focused on the geometric definitions behind continuous lattice path enumeration. Our most telling result is that through a limiting procedure with Todd operators, we are 
able to turn results about continuous Catalan numbers into results about discrete Catalan numbers. Therefore, studying the continuous case will lead to new insight about the discrete case. In this current paper, we therefore ignore the geometric intuition underlying the continuous binomial coefficients and Catalan numbers and treat them as analytic objects of independent interest.

We already have the fundamental result:
\begin{thm}\cite[Thm. 14]{Diaz2}
\label{cb1}
For $0\leq s \leq x$, the continuous binomial coefficient $\cbinom{x}{s}$ satisfies
\begin{equation}
\cbinom{x}{s}=2 I_{0}\left(\sqrt{s(x-s)}\right)+\frac{x}{\sqrt{s(x-s)}}I_{1}\left(2\sqrt{s(x-s)}\right),\label{eq:Catalan explicit2}
\end{equation}
where $I_{\nu}(z)$ denotes the modified Bessel function of the first kind.
\end{thm}
We prove the following closed form expression for continuous Catalan numbers in Section \ref{section2}.
\begin{thm}
\label{catalan1} 
The continuous Catalan numbers defined in \cite{Diaz2} are equal to
\begin{equation}
C\left(x,y\right)=I_{0}\left(\sqrt{x^{2}-y^{2}}\right)-\frac{x-y}{x+y}I_{2}\left(\sqrt{x^{2}-y^{2}}\right),\label{eq:Catalan explicit2}
\end{equation}
\end{thm}	
We can regard these expressions as \textit{definitions} for both objects, and indeed they lead to easy analytic continuations. The vast literature surrounding Bessel functions then means that we can prove several deep results about these two quantities, which should translate into new intuition about the discrete cases. We prove analogues of several discrete identities, such as the Chu-Vandermonde identity or Catalan convolution, and collect some integral transforms associated with both objects.

Moreover, we can naturally define, as in \cite{Diaz2}, the \textbf{continuous binomial distribution} $CB(x,p)$ associated to continuous binomials. It has parameters $x \geq 0$ and $0 \leq p \leq 1$, and density function
\begin{equation} 
		f_{x,p}(s) :=  \frac{1}{A_{x,p}} \cbinom{x}{s} p^{s}\left(1-p\right)^{x-s}, \label{eq:binomial}
\end{equation}
with $s\in[0,x]$. The normalization constant $A_{x,p}$ is such that
\[ \int_0^x f_{x,p}(s)\thinspace ds = 1, \]
and its value is given in Thm. \ref{normalization}.
	
	Sections \ref{section3} and \ref{section4} lead to several convolution identities and integral transforms for the continuous binomial coefficient, along with closed form expression for $A_{x,p}$ and the moment generating function for $f_{x,p}(s)$. Finally, in Section \ref{section5} we are able to prove a probabilistic interpretation of the continuous binomial coefficient due to its close connection to the Goldstein-Kac telegraph process.



\section{Continuous Catalan Numbers}\label{section2}
\subsection{Closed form}
Recall that the discrete Catalan numbers $C_{n}$ count the number of lattice paths joining the points $(0,0)$  and $(n,0)$ that stay above the $x$-axis.  Therefore, the continuous Catalan numbers must satisfy similar restrictions - they correspond to continuous analogues of lattice paths that stay above the $x$ axis.  


Let us first define the following polytope, which contains all possible paths in the plane made out of $n$ steps of arbitrary lengths in the East or North directions that connect the origin to the point $\left( \frac{x+y}{2},\frac{x-y}{2} \right)$  and remain under the line $y=x$.

\begin{defn}
For $n \geq 1$, the convex polytope $\Lambda^n(x,y)$ is defined as the set of all $(x_1, \dots, x_n, y_1, \dots, y_n) \in \R^{2n}$ that satisfy the following inequalities:
\begin{align*}
0 \leq x_1 \leq \cdots \leq x_n   &    \leq \frac{x+y}{2} \\
0 \leq y_1 \leq \cdots \leq y_n   &    \leq \frac{x-y}{2}  \\
 y_i  &  \leq x_i.
\end{align*}
\end{defn}

This polytope allows to define the continuous Catalan numbers as follows.
\begin{defn}\cite[Defn. 23]{Diaz2} The continuous Catalan numbers are defined by
\begin{equation}
C\left(x,y\right) :=\sum_{n=0}^{\infty}\vol\left(\Lambda^{n}\left(x,y\right)\right),\thinspace\thinspace0\le y\le x,            \label{eq:Catalan}
\end{equation}
where the volume is computed with respect to the Lebesgue measure.
\end{defn}
The volume of each of these polytopes can then be explicitly computed as follows.

\begin{lem}
For $n\ge0,$ the volume of $\Lambda^{n}\left(x,y\right)$
is equal to
\begin{equation}
\vol\left(\Lambda^{n}\left(x,y\right)\right)=\frac{\left(x-y\right)^{n}\left(x+y\right)^{n-1}\left(x+\left(2n+1\right)y\right)}{2^{2n}n!\left(n+1\right)!}.\label{eq:volume}
\end{equation}
\end{lem}
\begin{proof}
The proof of this lemma is elementary, and follows by induction on $n$. Namely, it is easily checked that the right-hand side satisfies the integral recurrrence \cite[Prop. 27]{Diaz2})
\[
\vol\left(\Lambda^{n+1}\left(x,y\right)\right)=\int_{0}^{\frac{x-y}{2}}\int_{0}^{\frac{x+y}{2}-b}\vol\left(\Lambda^{n}\left(a+2b,a\right)\right)\thinspace da\thinspace db,
\]
together with the initial condition $\vol\left(\Lambda^{0}\left(x,y\right)\right)=1.$
\end{proof}
The proof of Theorem \ref{catalan1} is now completed
by computing the sum in (\ref{eq:Catalan}) using the expression (\ref{eq:volume}).

Moreover, by summing \eqref{eq:volume} over all $n$, it is shown in \cite[Prop. 29]{Diaz2} that the continuous Catalan numbers obey the recursion
\begin{equation}
C(x,y)= 1+\int_{0}^{\frac{x-y}{2}}\int_{0}^{\frac{x+y}{2}-b}C(a+2b,a)\thinspace da\thinspace db.
\end{equation}
As a check, we can manually verify that the closed form \eqref{eq:Catalan} for the continuous Catalan numbers obeys the same recursion.

\subsection{Parallels between the continuous and discrete case}
The special case $y=0$ gives the continuous Catalan function as defined
in \cite{Diaz2}:
\[
C\left(2x,0\right)=\frac{I_{1}\left(2x\right)}{x}.
\]
With $C_{n}$ denoting the usual Catalan numbers, we observe that
\[
\frac{I_{1}\left(2x\right)}{x}=\sum_{n\ge0}\frac{x^{2n}}{n!\left(n+1\right)!}=\sum_{n\ge0}\frac{x^{2n}}{2n!}C_{n}.
\]
Therefore, the continuous Catalan function $C\left(2x,0\right)$ is related
to the generating function of Catalan numbers by
\begin{equation}
\label{laplaceC}
\int_{0}^{+\infty}C\left(2\sqrt{x}u,0\right)e^{-u}du=\sum_{n}C_{n}x^{n}=\frac{2}{1+\sqrt{1-4x^{2}}}.
\end{equation}

The fact that a discrete generating function of the Catalan numbers
is related to a continuous integral transform of the continuous Catalan
function should not come as a surprise. However, the fact that the continuous Catalan numbers have a simple closed form expression in terms of Bessel functions lends hope to discovering closed form expressions for the continuous analogues of other objects that count lattice paths, such as the Delannoy numbers. 

A further parallel between the classical and continuous cases is provided by considering the convolution identity
\begin{equation}\label{catalanconvdisc}
\sum_{k=0}^{n-1}C_{k}C_{n-k}=\Delta C_{n}
\end{equation}
with $\Delta C_{n}=C_{n+1}-C_{n},$
a consequence of the fact that the generating function
\[
c\left(z\right)=\sum_{n\ge0}C_{n}z^{n}=\frac{2}{1+\sqrt{1-4z}}
\]
satisfies the equation
\[
c^{2}\left(z\right)=\frac{c\left(z\right)-1}{z}.
\]
For the continuous Catalan numbers, we have the following result, which is clearly a continuous analogue of the discrete identity \eqref{catalanconvdisc}.

\begin{thm}
The continuous Catalan numbers $C\left(x\right)=C\left(x,0\right)$ satisfy the convolution identity
\begin{equation}
\int_{0}^{z}C\left(x\right)C\left(z-x\right)dx=4\frac{d}{dz}C\left(z\right).\label{eq:continuous-convolution}
\end{equation}
\end{thm}
\begin{proof}
Since
\[
C\left(x\right)=C\left(x,0\right)=2\frac{I_{1}\left(x\right)}{x},
\]
the continuous equivalent of this generating function is the Laplace
transform
\[
\mathcal{L}_{C}\left(s\right)=\int_{0}^{+\infty}C\left(x\right)e^{-xs}dx
=\frac{2}{s+\sqrt{s^{2}-1}}.
\]
Since the derivative of the continuous Catalan number
\[
C'\left(x\right) = \frac{d}{dx}C\left(x\right) = 2\frac{I_2\left(x\right)}{x}
\]
has Laplace transform
\[
\mathcal{L}_{C'}\left(s\right) = -1+2s^2-2s\sqrt{s^2-1},
\]
we deduce the identity
\[
\mathcal{L}_{C}^{2}\left(s\right)=4\mathcal{L}_{C'}\left(s\right).
\]
Taking the
inverse Laplace transform of this identity gives the desired identity.
\end{proof}

\subsection{Integral representations}We calculate some useful integral representations for the continuous Catalan numbers, which enable the easy application of Laplace-transformation type proofs. These also allow us to view the continuous Catalan numbers as probability distribution functions, and recover various moment expressions for the discrete Catalan numbers.

\begin{thm}
We have the integral representations
\[
C\left(x,y\right)=\frac{1}{\pi}\int_{0}^{\pi}e^{x\cos t}\left[\cos\left(y\sin t\right)-\left(\frac{x-y}{x+y}\right)^{2}\cos\left(y\sin t-2t\right)\right]dt
\]
and
\begin{equation*}
C\left(2x,0\right)=\frac{2}{\pi}\int_{0}^{\pi}e^{2x\cos t}\sin^{2}tdt=\frac{I_{1}\left(2x\right)}{x}.\label{eq:C integral representation}
\end{equation*}

\end{thm}
\begin{proof}
This follows from a straightforward application of the generalized Schl\"afli formula \cite[p. 81]{Magnus}, 
\begin{align*}
\left(\frac{a-b}{a+b}\right)^{-\frac{\nu}{2}}J_{\nu}\left(\sqrt{a^{2}-b^{2}}\right) & =\frac{1}{\pi}\int_{0}^{\pi}e^{b\cos t}\cos\left(a\sin t-\nu t\right)dt\\
 & -\frac{\sin\left(\pi\nu\right)}{\pi}\int_{0}^{+\infty}e^{-a\sinh t-b\cosh t-\nu t}dt.
\end{align*}
This is transfered to the Bessel $I$ functions using $I_{\nu}\left(z\right)=e^{-\imath\nu\frac{\pi}{2}}J_{\nu}\left(\imath z\right).$ 

For $\nu=0$ we have
\[
I_{0}\left(\sqrt{a^{2}-b^{2}}\right)=\frac{1}{\pi}\int_{0}^{\pi}e^{a\cos t}\cos\left(b\sin t\right)dt,
\]

for $\nu=1$,
\[
I_{1}\left(\sqrt{a^{2}-b^{2}}\right)=\sqrt{\frac{a-b}{a+b}}\frac{1}{\pi}\int_{0}^{\pi}e^{a\cos t}\cos\left(b\sin t-t\right)dt,
\]

and for $\nu=2$,
\[
I_{2}\left(\sqrt{a^{2}-b^{2}}\right)=\frac{a-b}{a+b}\frac{1}{\pi}\int_{0}^{\pi}e^{a\cos t}\cos\left(b\sin t-2t\right)dt.
\]
Substituting these into the closed form expression from Theorem \ref{catalan1} completes the proof.
\end{proof}

From this integral representation, we easily recover the expression \eqref{laplaceC}
\begin{align*}
\int_{0}^{+\infty}e^{-u}C\left(2u\sqrt{x},0\right)du & =\int_{0}^{+\infty}e^{-u}\frac{2}{\pi}\int_{0}^{\pi}e^{2u\sqrt{x}\cos t}\sin^{2}tdtdu\\
 & =\frac{2}{\pi}\int_{0}^{\pi}\sin^{2}t\int_{0}^{+\infty}e^{-u}e^{2u\sqrt{x}\cos t}dudt\\
 & =\frac{2}{\pi}\int_{0}^{\pi}\frac{\sin^{2}t}{1-2\sqrt{x}\cos t}dt=\frac{2}{1+\sqrt{1-4x}}.
\end{align*}
The change of variable $z=\cos t$  in the second integral representation in Theorem \ref{eq:C integral representation} also gives
\[
C\left(x,0\right)=\frac{2}{\pi}\int_{-1}^{+1}e^{xz}\sqrt{1-z^{2}}dz,
\]
which can be expressed as
\[
C\left(x,0\right) = \mathbb{E}e^{Zx},
\]
the moment generating function of a  random variable $Z$ distributed according to the semi-circle distribution $\frac{2}{\pi}\sqrt{1-x^{2}}$. This is the continuous
equivalent of the representation of Catalan numbers as the moments of the same distribution,
\[
C_{n}=\frac{2}{\pi}\int_{-1}^{1}\left(2z\right)^{2n}\sqrt{1-z^{2}}dz.
\]


We now prove a general integral formula involving $C(x,y)$. An analogue of this formula for continuous binomial coefficients is the key element in our analysis of the continuous binomial distribution.
\begin{thm}
Given any function $\Phi\left(y\right)$ supported on $[0,x]$, the integral
\[
I_{\Phi}\left(x\right)=\int_{0}^{x}C\left(x,y\right)\Phi\left(y\right)dy,
\]
can be computed as 
\begin{align*}
I_{\Phi}\left(x\right) =\frac{1}{\pi}\int_{-1}^{+1}e^{-xu}\hat{\Phi}\left(u\right)\left(\sqrt{1+\frac{1}{u}}-1\right)\frac{du}{\sqrt{1-u^{2}}},
\end{align*}
where $$\hat{\Phi}(t) = \int_{0}^{x}\Phi\left(y\right)e^{-\imath y t}dy$$ is the Fourier transform of $\Phi$.
\end{thm}
\begin{proof}
Consider a function $\Phi\left(y\right)$ with support $[0,x]$ and the integral
\[
I_{\Phi}\left(x\right)=\int_{0}^{x}C\left(x,y\right)\Phi\left(y\right)dy.
\]
Then the Laplace transform $\tilde{I}_{\Phi}\left(p\right)$ of $I_{\Phi}$
can be computed as 
\begin{align*}
\tilde{I}_{\Phi}\left(p\right) & =\int_{0}^{+\infty}I_{\Phi}\left(x\right)e^{-px}dx=\int_{0}^{+\infty}\int_{0}^{x}C\left(x,y\right)\Phi\left(y\right)dye^{-px}dx\\
 & =\int_{0}^{+\infty}\Phi\left(y\right)\int_{y}^{+\infty}e^{-px}C\left(x,y\right)dxdy.
\end{align*}
We then exploit the Laplace transforms \cite[3.15.4.2, 3.15.4.9]{Prudnikov 4}
\begin{align*}
\int_{y}^{+\infty}e^{-px}I_{0}\left(\sqrt{x^{2}-y^{2}}\right)dx&=\frac{1}{\sqrt{p^{2}-1}}e^{-y\sqrt{p^{2}-1}}, \\
\int_{y}^{+\infty}e^{-px}\frac{x-y}{x+y}I_{2}\left(\sqrt{x^{2}-y^{2}}\right)dx & =\int_{y}^{+\infty}e^{-px}\frac{x^{2}-y^{2}}{\left(x+y\right)^{2}}I_{2}\left(\sqrt{x^{2}-y^{2}}\right)dx\\
 & =\frac{1}{\left(p+\sqrt{p^{2}-1}\right)^{2}}\frac{1}{\sqrt{p^{2}-1}}e^{-y\sqrt{p^{2}-1}}
\end{align*}
to deduce
\begin{align*}
\tilde{I}_{\Phi}\left(p\right) & =\int_{0}^{+\infty}\Phi\left(y\right)\frac{1}{\sqrt{p^{2}-1}}e^{-y\sqrt{p^{2}-1}}\left[1+\frac{1}{\left(p+\sqrt{p^{2}-1}\right)^{2}}\right]dy\\
& =\frac{2p}{p+\sqrt{p^{2}-1}}\frac{1}{\sqrt{p^{2}-1}}\tilde{\Phi}\left(\sqrt{p^{2}-1}\right),
\end{align*}
with $\tilde{\Phi}$ the Laplace transform of $\Phi.$
Since
\[
\sum_{n\ge0}\frac{C_{n}}{\left(4p\right)^{n}}=\frac{2p}{p+\sqrt{p^{2}-1}},
\]
we can write
\[
\tilde{I}_{\Phi}\left(p\right)=\sum_{n\ge0}\frac{C_{n}}{2^{2n}}\frac{1}{p^{n}}\frac{\tilde{\Phi}\left(\sqrt{p^{2}-1}\right)}{\sqrt{p^{2}-1}}
\]
and the corresponding inverse Laplace transform
\[
I_{\Phi}\left(x\right)=\sum_{n\ge0}\frac{C_{n}}{2^{2n}}\mathcal{L}^{-1}\left[\frac{1}{p^{n}}\frac{\tilde{\Phi}\left(\sqrt{p^{2}-1}\right)}{\sqrt{p^{2}-1}}\right].
\]
Now we can use the results
\[
\mathcal{L}^{-1}\left[\frac{\tilde{\Phi}\left(\sqrt{p^{2}-1}\right)}{\sqrt{p^{2}-1}}\right]=\int_{0}^{x}I_{0}\left(\sqrt{x^{2}-y^{2}}\right)\Phi\left(y\right)dy
\]
and
\[
\mathcal{L}^{-1}\left[\frac{F\left(p\right)}{p^{n}}\right]=f^{\left(-n\right)}\left(t\right),
\]
where $f^{\left(-n\right)}$ is the $n-$th antiderivative of $f.$ Starting from the integral representation
\[
I_{0}\left(\sqrt{x^{2}-y^{2}}\right)=\frac{1}{2\pi}\int_{0}^{2\pi}e^{-\left(x\cos\theta+\imath y\cos\theta\right)}d\theta,
\]
we have
\begin{align*}
\int_{0}^{x}I_{0}\left(\sqrt{x^{2}-y^{2}}\right)\Phi\left(y\right)dy & =\int_{0}^{x}\frac{1}{2\pi}\int_{0}^{2\pi}e^{-\left(x\cos\theta+\imath y\cos\theta\right)}d\theta\Phi\left(y\right)dy\\
 & =\frac{1}{2\pi}\int_{0}^{2\pi}e^{-x\cos\theta}d\theta\int_{0}^{x}\Phi\left(y\right)e^{-\imath y\cos\theta}dy.
\end{align*}
Since the function $\Phi\left(y\right)$ has bounded support on
$\left[0,x\right],$ the inner integral is recognized as its Fourier
transform $\hat{\Phi}$ computed at $\cos\theta,$ and
\[
\int_{0}^{x}I_{0}\left(\sqrt{x^{2}-y^{2}}\right)\Phi\left(y\right)dy=\frac{1}{2\pi}\int_{0}^{2\pi}e^{-x\cos\theta}\hat{\Phi}\left(\cos\theta\right)d\theta.
\]
The antiderivative of order $n$ in $x$ is easily computed as
\[
\frac{\left(-1\right)^{n}}{2\pi}\int_{0}^{2\pi}\frac{e^{-x\cos\theta}}{\left(\cos\theta\right)^{n}}\hat{\Phi}\left(\cos\theta\right)d\theta.
\]
Consequently,
\begin{align*}
I_{\Phi}\left(x\right) & =\sum_{n\ge0}\frac{C_{n}}{2^{2n}}\frac{\left(-1\right)^{n}}{2\pi}\int_{0}^{2\pi}\frac{e^{-x\cos\theta}}{\left(\cos\theta\right)^{n}}\hat{\Phi}\left(\cos\theta\right)d\theta\\
 & =\frac{1}{2\pi}\int_{0}^{2\pi}e^{-x\cos\theta}\hat{\Phi}\left(\cos\theta\right)2\cos\theta\left(-1+\sqrt{1+\frac{1}{\cos\theta}}\right)d\theta\\
 & =\frac{1}{\pi}\int_{-1}^{+1}e^{-xu}\hat{\Phi}\left(u\right)\left(\sqrt{1+\frac{1}{u}}-1\right)\frac{du}{\sqrt{1-u^{2}}},
\end{align*}
which is the desired result.
\end{proof}

\section{Continuous Binomial Coefficients} \label{section3}

\subsection{Integrals}
Here we examine some of the properties of the continuous binomials $\cbinom{x}{s}$, including some integral transforms that will allow us to prove several more general theorems in Sections \ref{section4} and \ref{section5}. We start with a general integral transform that will later appear in the analysis of the continuous binomial distribution.
\begin{thm}\label{laplace}
The function
\[
J_{\Phi}\left(x\right)=\int_{0}^{x}\cbinom{x}{s} \Phi\left(s\right)ds
\]
has Laplace transform
\[
\int_{0}^{+\infty}J_{\Phi}\left(x\right)e^{-px}dx=\left(\frac{1+p}{p}\right)^{2}\tilde{\Phi}\left(p-\frac{1}{p}\right)-\tilde{\Phi}\left(p\right),
\]
where $\tilde{\Phi}\left(p\right)$ is the Laplace transform of $\Phi\left(s\right).$ 

As a consequence,
\begin{equation}
J_{\Phi}\left(x\right)=\mathcal{L}^{-1}\left[\left(\frac{1+p}{p}\right)^{2}\tilde{\Phi}\left(p-\frac{1}{p}\right)\right]-\Phi\left(x\right)\label{eq:Inverse Laplace},
\end{equation}
where $\mathcal{L}^{-1}$ denotes the inverse Laplace transform.
\end{thm}
\begin{proof}
We apply Fubini's theorem to transform the double integral
\begin{align*}
\int_{0}^{+\infty}J_{\Phi}\left(x\right)e^{-px}dx & =\int_{0}^{+\infty}\int_{0}^{x}\cbinom{x}{s}\Phi\left(s\right)ds\thinspace\thinspace e^{-px}dx\\
 & =\int_{0}^{+\infty}\Phi\left(s\right)\int_{s}^{+\infty}e^{-px}\cbinom{x}{s} dx\thinspace\thinspace ds.
\end{align*}
The inner integral is now evaluated using the change of variable $x=s+w$
as
\[
\int_{s}^{+\infty}e^{-px}\cbinom{x}{s}dx=e^{-sp}\int_{0}^{+\infty}e^{-pw}\cbinom{s+w}{s} dw.
\]
Using the closed form \eqref{cb1} for the continuous binomial coefficient, we deduce
\begin{align*}
\int_{0}^{+\infty}e^{-pw}\cbinom{s+w}{s} dw 
 & =2\int_{0}^{+\infty}e^{-pw}I_{0}\left(2\sqrt{sw}\right)dw+\int_{0}^{+\infty}e^{-pw}\frac{w+s}{\sqrt{sw}}I_{1}\left(2\sqrt{sw}\right)dw.
\end{align*}
These Laplace transforms can be found in \cite[6.614.3 and 6.643.2]{Gradshteyn}
and evaluate to
\[
\int_{0}^{+\infty}e^{-pw}I_{0}\left(2\sqrt{sw}\right)dw=\frac{1}{p}e^{\frac{s}{p}},
\]
\begin{align*}
\int_{0}^{+\infty}e^{-pw}\frac{w}{\sqrt{sw}}I_{1}\left(2\sqrt{sw}\right)dw & =\frac{1}{p^{2}}e^{\frac{s}{p}},
\end{align*}
and
\[
\int_{0}^{+\infty}e^{-pw}\frac{1}{\sqrt{sw}}I_{1}\left(2\sqrt{sw}\right)dw=-1+e^{\frac{s}{p}}.
\]
We deduce
\begin{equation}
\label{Laplace_CB}
\int_{0}^{+\infty}e^{-pw}\cbinom{s+w}{s} dw=e^{\frac{s}{p}}\frac{p^{2}+2p+1}{p^{2}}-1=e^{\frac{s}{p}}\left(\frac{p+1}{p}\right)^{2}-1.
\end{equation}
This is now substituted in the outer integral to obtain
\[
\int_{0}^{+\infty}\Phi\left(s\right)e^{-sp}\left\{ e^{\frac{s}{p}}\left(\frac{p+1}{p}\right)^{2}-1\right\} ds=\left(\frac{p+1}{p}\right)^{2}\int_{0}^{+\infty}\Phi\left(s\right)e^{-s\left(p-\frac{1}{p}\right)}ds-\int_{0}^{+\infty}\Phi\left(s\right)e^{-sp}ds.
\]
These two integral are recognized as the Laplace transforms of $\Phi\left(s\right)$
computed respectively at $p-\frac{1}{p}$ and $p,$ and the result
follows.
\end{proof}

There are several important special cases.
\begin{cor}
Choosing $\Phi\left(s\right)=\alpha^{s}e^{us},$ we deduce the value
of the integral 
\begin{align}
\int_{0}^{x}\cbinom{x}{s}\alpha^{s}e^{us}ds & =2\alpha^{\frac{x}{2}}e^{\frac{ux}{2}}\left\{ \cosh\left(\frac{x}{2}\sqrt{4+\left(u+\log\alpha\right)^{2}}\right)-\cosh\left(\frac{x}{2}\left(u+\log\alpha\right)\right)\right.\label{eq:alpha^se^us}\\
 & +\left.\frac{2}{\sqrt{4+\left(u+\log\alpha\right)^{2}}}\sinh\left(\frac{x}{2}\sqrt{4+\left(u+\log\alpha\right)^{2}}\right)\right\} .\nonumber 
\end{align}
\end{cor}
\begin{proof}
Choose $\Phi\left(s\right)=\alpha^{s}e^{us}$ so that $\tilde{\Phi}\left(p\right)=\frac{1}{p-u-\log\alpha}$
and use formula (\ref{eq:Inverse Laplace}) to obtain the result.
\end{proof}
Another consequence is as follows.
\begin{cor}
The integral
\begin{equation}
\int_{0}^{x}\cbinom{x}{s}ds=2\left(e^{x}-1\right)
\end{equation}
holds, as computed in \cite{Diaz2}. 
\end{cor}

\subsection{Continuous Chu-Vandermonde formula}
Now that we have obtained continuous binomial coefficients with nice reductions to the discrete case, we can try to find continuous generalizations of discrete identities. We first consider an averaged case of the Chu-Vandermonde identity, 
\[
\sum_{k}\binom{k+s_{1}}{s_{1}}\binom{n-k+s_{2}}{s_{2}}=\binom{n+s_{1}+s_{2}+1}{n},
\]
which can be regarded as the prototypical binomial convolution. We denote the Dirac delta distribution by $\delta(x)$ and by $*$ the integral convolution 
\[
f*g\left(x\right) = \int f\left(u\right)g\left(x-u\right)du
\]
and notice that $f*\delta = f.$
\begin{thm}
\label{Thm11}
The continuous binomial coefficient satisfies the  identity
\[
\left(\delta\left(x\right)+\left\{ \begin{array}{c}
x+s_{1}\\
s_{1}
\end{array}\right\} \right)*\left(\delta\left(x\right)+\left\{ \begin{array}{c}
x+s_{2}\\
s_{2}
\end{array}\right\} \right)=\left(\delta\left(x\right)+\left\{ \begin{array}{c}
x+\frac{s_{1}+s_{2}}{2}\\
\frac{s_{1}+s_{2}}{2}
\end{array}\right\} \right)*\left(\delta\left(x\right)+\left\{ \begin{array}{c}
x+\frac{s_{1}+s_{2}}{2}\\
\frac{s_{1}+s_{2}}{2}
\end{array}\right\} \right),
\]
to be compared to the discrete version
\[
\sum_{k}\binom{k+s_{1}}{s_{1}}\binom{n-k+s_{2}}{s_{2}}
=\sum_{k}\binom{k+\frac{s_{1}+s_{2}}{2}}{\frac{s_{1}+s_{2}}{2}}\binom{n-k+\frac{s_{1}+s_{2}}{2}}{\frac{s_{1}+s_{2}}{2}}.
\]
\end{thm}
\begin{proof}
From \eqref{Laplace_CB}, the Laplace transform (in the variable $x$) of the continuous binomial
coefficient
\[
f_{s}\left(x\right)=\left\{ \begin{array}{c}
x+s\\
s
\end{array}\right\} 
\]
is
\[
F_{s}\left(p\right)=\int_{0}^{+\infty}e^{-xp}\left\{ \begin{array}{c}
x+s\\
s
\end{array}\right\} dx=e^{\frac{s}{p}}\left(\frac{p+1}{p}\right)^{2}-1.
\]
We deduce
\begin{align*}
F_{s_{1}}\left(p\right)F_{s_{2}}\left(p\right) & =\left(e^{\frac{s_{1}}{p}}\left(\frac{p+1}{p}\right)^{2}-1\right)\left(e^{\frac{s_{2}}{p}}\left(\frac{p+1}{p}\right)^{2}-1\right)\\
 & =e^{\frac{s_{1}+s_{2}}{p}}\left(\frac{p+1}{p}\right)^{4}-1-\left(e^{\frac{s_{1}}{p}}\left(\frac{p+1}{p}\right)^{2}-1\right)-\left(e^{\frac{s_{2}}{p}}\left(\frac{p+1}{p}\right)^{2}-1\right).
\end{align*}
The decomposition
\[
e^{\frac{s_{1}+s_{2}}{p}}\left(\frac{p+1}{p}\right)^{4}-1=\left(e^{\frac{s_{1}+s_{2}}{2p}}\left(\frac{p+1}{p}\right)^{2}-1\right)\left(\left(e^{\frac{s_{1}+s_{2}}{2p}}\left(\frac{p+1}{p}\right)^{2}-1\right)+2\right)
\]
gives
\[
F_{s_{1}}\left(p\right)F_{s_{2}}\left(p\right)=F_{\frac{s_{1}+s_{2}}{2}}\left(p\right)\left(F_{\frac{s_{1}+s_{2}}{2}}\left(p\right)+2\right)-F_{s_{1}}\left(p\right)-F_{s_{2}}\left(p\right),
\]
the inverse Laplace transform of which is
\[
\left(f_{s_{1}}*f_{s_{2}}\right)\left(x\right)=\left(f_{\frac{s_{1}+s_{2}}{2}}*f_{\frac{s_{1}+s_{2}}{2}}\right)\left(x\right)+2f_{\frac{s_{1}+s_{2}}{2}}\left(x\right)-f_{s_{1}}\left(x\right)-f_{s_{2}}\left(x\right).
\]
Equivalently,
\begin{align}
\left\{ \begin{array}{c}
x+s_{1}\\
s_{1}
\end{array}\right\} *\left\{ \begin{array}{c}
x+s_{2}\\
s_{2}
\end{array}\right\} +\left\{ \begin{array}{c}
x+s_{1}\\
s_{1}
\end{array}\right\} +\left\{ \begin{array}{c}
x+s_{2}\\
s_{2}
\end{array}\right\}  & =\left\{ \begin{array}{c}
x+\frac{s_{1}+s_{2}}{2}\\
\frac{s_{1}+s_{2}}{2}
\end{array}\right\} *\left\{ \begin{array}{c}
x+\frac{s_{1}+s_{2}}{2}\\
\frac{s_{1}+s_{2}}{2}
\end{array}\right\} \label{eq:binomial convolution}\\
 & +2\left\{ \begin{array}{c}
x+\frac{s_{1}+s_{2}}{2}\\
\frac{s_{1}+s_{2}}{2}
\end{array}\right\} .\nonumber 
\end{align}
The theorem follows after rewriting this identity in terms of Dirac delta functions.
\end{proof}


We can then give an analogue of the Chu-Vandermonde identity, based on discrete difference and differential operators. We begin with the discrete case: define the $\star$ operator as 
\begin{equation}
\label{star}
\binom{n+k_{1}}{k_{1}} \star \binom{n+k_{2}}{k_{2}}=\sum_{m=1}^{n-1}\binom{m+k_{1}}{k_{1}}\binom{n-m+k_{2}}{k_{2}},
\end{equation}
so that the Chu-Vandermonde identity reads
\[
\binom{n+k_{1}}{k_{1}} \star \binom{n+k_{2}}{k_{2}}+\binom{n+k_{1}}{k_{1}}+\binom{n+k_{2}}{k_{2}}=\binom{n+k_{1}+k_{2}+1}{k_{1}+k_{2}+1}=\left(1+\Delta_{k_{1}+k_{2}}\right)\binom{n+k_{1}+k_{2}}{k_{1}+k_{2}},
\]
where $\Delta_k$ is the forward discrete difference operator in the variable $k,$
\[
\Delta_k f\left(n+k\right) = f\left(n+k+1\right).
\]
Its continuous analogue is as follows.
\begin{thm}
With $\bar{s}=s_{1}+s_{2},$  the continuous binomial coefficient 
satisfies the differential equation
\[
\left\{ \begin{array}{c}
x+s_{1}\\
s_{1}
\end{array}\right\} *\left\{ \begin{array}{c}
x+s_{2}\\
s_{2}
\end{array}\right\} +\left\{ \begin{array}{c}
x+s_{1}\\
s_{1}
\end{array}\right\} +\left\{ \begin{array}{c}
x+s_{2}\\
s_{2}
\end{array}\right\} =\left(1+\frac{\partial}{\partial\bar{s}}\right)^{2}\left\{ \begin{array}{c}
x+\bar{s}\\
\bar{s}
\end{array}\right\} .
\]
\end{thm}

\begin{proof}
Applying Thm. \ref{Thm11}, we have 
\[
\left(\left\{ \begin{array}{c}
x+s_{1}\\
s_{1}
\end{array}\right\} +\delta\left(x\right)\right)*\left(\left\{ \begin{array}{c}
x+s_{2}\\
s_{2}
\end{array}\right\} +\delta\left(x\right)\right)=\left(\left\{ \begin{array}{c}
x+\bar{s}\\
\bar{s}
\end{array}\right\} +\delta\left(x\right)\right)*\left(\left\{ \begin{array}{c}
x\\
0
\end{array}\right\} +\delta\left(x\right)\right),
\]
where
\[
\left\{ \begin{array}{c}
x\\
0
\end{array}\right\} =x+2
\]
is deduced from  the Laplace transform
\[
\mathcal{L} \left( \left\{ \begin{array}{c}
x\\
0
\end{array}\right\} +\delta\left(x\right) \right)
=\left(1+\frac{1}{p}\right)^{2}
=\mathcal{L}\left(2+x+\delta\left(x\right)\right).
\]
We thus need to compute the convolution
\[
\left\{ \begin{array}{c}
x+\bar{s}\\
\bar{s}
\end{array}\right\} *\left(2+x\right)=\int_{0}^{x}\left(2+x-u\right)\left\{ \begin{array}{c}
u+\bar{s}\\
u
\end{array}\right\} du.
\]
First, using the differential equation
\[
\frac{\partial}{\partial u}\frac{\partial}{\partial\bar{s}}\left\{ \begin{array}{c}
u+\bar{s}\\
u
\end{array}\right\} =\left\{ \begin{array}{c}
u+\bar{s}\\
u
\end{array}\right\} ,
\]
we deduce
\begin{align*}
\int_{0}^{x}\left\{ \begin{array}{c}
u+\bar{s}\\
u
\end{array}\right\} du & =\int_{0}^{x}\frac{\partial}{\partial u}\frac{\partial}{\partial\bar{s}}\left\{ \begin{array}{c}
u+\bar{s}\\
u
\end{array}\right\} du=\frac{\partial}{\partial\bar{s}}\int_{0}^{x}\frac{\partial}{\partial u}\left\{ \begin{array}{c}
u+\bar{s}\\
u
\end{array}\right\} du\\
 & =\frac{\partial}{\partial\bar{s}}\left(\left\{ \begin{array}{c}
x+\bar{s}\\
x
\end{array}\right\} -\left\{ \begin{array}{c}
\bar{s}\\
\bar{s}
\end{array}\right\} \right)=\frac{\partial}{\partial\bar{s}}\left\{ \begin{array}{c}
x+\bar{s}\\
x
\end{array}\right\} -1.
\end{align*}
This argument also shows that an antiderivative of $\left\{ \begin{array}{c}
u+\bar{s}\\
u
\end{array}\right\} $ is $\frac{\partial}{\partial\bar{s}}\left\{ \begin{array}{c}
x+\bar{s}\\
x
\end{array}\right\} .$ Next, integrating by parts gives
\begin{align*}
\int_{0}^{x}u\left\{ \begin{array}{c}
u+\bar{s}\\
u
\end{array}\right\} du & =\left[u\frac{\partial}{\partial\bar{s}}\left\{ \begin{array}{c}
x+\bar{s}\\
x
\end{array}\right\} \right]_{0}^{x}-\int_{0}^{x}\frac{\partial}{\partial\bar{s}}\left\{ \begin{array}{c}
x+\bar{s}\\
x
\end{array}\right\} du\\
 & =x\frac{\partial}{\partial\bar{s}}\left\{ \begin{array}{c}
x+\bar{s}\\
x
\end{array}\right\} -\frac{\partial^{2}}{\partial\bar{s}^{2}}\left\{ \begin{array}{c}
x+\bar{s}\\
x
\end{array}\right\} .
\end{align*}
We deduce
\begin{align*}
\int_{0}^{x}\left(2+x-u\right)\left\{ \begin{array}{c}
u+\bar{s}\\
u
\end{array}\right\} du & =\left(2+x\right)\left(\frac{\partial}{\partial\bar{s}}\left\{ \begin{array}{c}
x+\bar{s}\\
x
\end{array}\right\} -1\right)-x\frac{\partial}{\partial\bar{s}}\left\{ \begin{array}{c}
x+\bar{s}\\
x
\end{array}\right\} +\frac{\partial^{2}}{\partial\bar{s}^{2}}\left\{ \begin{array}{c}
x+\bar{s}\\
x
\end{array}\right\} \\
 & =-\left(x+2\right)+2\frac{\partial}{\partial\bar{s}}\left\{ \begin{array}{c}
x+\bar{s}\\
x
\end{array}\right\} +\frac{\partial^{2}}{\partial\bar{s}^{2}}\left\{ \begin{array}{c}
x+\bar{s}\\
x
\end{array}\right\} .
\end{align*}
Finally, we deduce the convolution
\begin{align*}
 & \left(\left\{ \begin{array}{c}
x+s_{1}\\
s_{1}
\end{array}\right\} +\delta\left(x\right)\right)*\left(\left\{ \begin{array}{c}
x+s_{2}\\
s_{2}
\end{array}\right\} +\delta\left(x\right)\right)\\
 & =\left(\left\{ \begin{array}{c}
x+\bar{s}\\
\bar{s}
\end{array}\right\} +\delta\left(x\right)\right)*\left(\left\{ \begin{array}{c}
x\\
0
\end{array}\right\} +\delta\left(x\right)\right)\\
 & =\left\{ \begin{array}{c}
x+\bar{s}\\
\bar{s}
\end{array}\right\} *\left\{ \begin{array}{c}
x\\
0
\end{array}\right\} +\left\{ \begin{array}{c}
x+\bar{s}\\
\bar{s}
\end{array}\right\} +\left\{ \begin{array}{c}
x\\
0
\end{array}\right\} +\delta\left(x\right)\\
 & =-\left(x+2\right)+2\frac{\partial}{\partial\bar{s}}\left\{ \begin{array}{c}
x+\bar{s}\\
x
\end{array}\right\} +\frac{\partial^{2}}{\partial\bar{s}^{2}}\left\{ \begin{array}{c}
x+\bar{s}\\
x
\end{array}\right\} +\left\{ \begin{array}{c}
x+\bar{s}\\
\bar{s}
\end{array}\right\} +\left\{ \begin{array}{c}
x\\
0
\end{array}\right\} +\delta\left(x\right)\\
 & =2\frac{\partial}{\partial\bar{s}}\left\{ \begin{array}{c}
x+\bar{s}\\
x
\end{array}\right\} +\frac{\partial^{2}}{\partial\bar{s}^{2}}\left\{ \begin{array}{c}
x+\bar{s}\\
x
\end{array}\right\} +\left\{ \begin{array}{c}
x+\bar{s}\\
\bar{s}
\end{array}\right\} +\delta\left(x\right)\\
 & =\left(1+\frac{\partial}{\partial\bar{s}}\right)^{2}\left\{ \begin{array}{c}
x+\bar{s}\\
\bar{s}
\end{array}\right\} +\delta\left(x\right).
\end{align*}
A more direct proof involves converting every term into the Laplace transform domain: start with
\[
\mathcal{L}\left(
\left(\left\{ \begin{array}{c}
x+s_{1}\\
s_{1}
\end{array}\right\} +\delta\left(x\right)\right)*\left(\left\{ \begin{array}{c}
x+s_{2}\\
s_{2}
\end{array}\right\} +\delta\left(x\right)\right) \right)= e^{\frac{s_{1}+s_{2}}{p}}\left(1+\frac{1}{p}\right)^{4}=e^{\frac{\bar{s}}{p}}\left(1+\frac{1}{p}\right)^{4}
\]
and
\[
\mathcal{L}\left(
\left\{ \begin{array}{c}
x+\bar{s}\\
\bar{s}
\end{array}\right\} +\delta\left(x\right)\right)= e^{\frac{\bar{s}}{p}}\left(1+\frac{1}{p}\right)^{2}.
\]
Since
\[
\frac{\partial}{\partial\bar{s}}e^{\frac{\bar{s}}{p}}=\frac{1}{p}e^{\frac{\bar{s}}{p}},
\]
it follows that
\[
\left(1+\frac{\partial}{\partial\bar{s}}\right)^{2}\left(\left\{ \begin{array}{c}
x+\bar{s}\\
\bar{s}
\end{array}\right\} +\delta\left(x\right)\right)=\left(1+\frac{\partial}{\partial\bar{s}}\right)^{2}\left\{ \begin{array}{c}
x+\bar{s}\\
\bar{s}
\end{array}\right\} +\delta\left(x\right)
\]
has Laplace transform
\[
e^{\frac{\bar{s}}{p}}\left(1+\frac{1}{p}\right)^{2}\left(1+\frac{1}{p}\right)^{2}=e^{\frac{\bar{s}}{p}}\left(1+\frac{1}{p}\right)^{4}.
\]
\end{proof}
This second proof also allows us to state the following generalization
\begin{cor}
With $\bar{s}=\sum_{i=1}^{p}s_{i},$ we have
\[
\left(\left\{ \begin{array}{c}
x+s_{1}\\
s_{1}
\end{array}\right\} +\delta\left(x\right)\right)*\dots*\left(\left\{ \begin{array}{c}
x+s_{p}\\
s_{p}
\end{array}\right\} +\delta\left(x\right)\right)=\left(1+\frac{\partial}{\partial\bar{s}}\right)^{2p}\left\{ \begin{array}{c}
x+\bar{s}\\
\bar{s}
\end{array}\right\} +\delta\left(x\right).
\]
\end{cor}

\subsection{Central binomial coefficients}

The central binomial coefficients $\left\{ \begin{array}{c}
2s\\
s
\end{array}\right\} $ have explicit expression
\[
\left\{ \begin{array}{c}
2s\\
s
\end{array}\right\} =2I_{0}\left(2s\right)+2I_{1}\left(2s\right)=\left(2+\frac{d}{ds}\right)I_{0}\left(2s\right),
\]
and Laplace transform
\[
\mathcal{L}\left(
\left\{ \begin{array}{c}
2s\\
s
\end{array}\right\} +\delta\left(s\right)\right) =\sqrt{\frac{p+2}{p-2}}.
\]

The parallel with the usual central binomial coefficients already appears in the asymptotic behavior: as it is well known, for large $n,$
\[
\frac{1}{2^{2n}}\binom{2n}{n} \sim \frac{1}{\sqrt{\pi n}} 
\]
whereas elementary asymptotic behavior results on Bessel I functions give, for large $s,$
\[
\frac{1}{2}e^{-4s} \left\{ 
\begin{array}{c}
2s\\
s
\end{array}\right\}
\sim \frac{1}{\sqrt{\pi s}}.
\]
The next theorem gives the continuous analogue of the convolution identity for central binomial coefficients
\[
\sum_{k=0}^{n}\binom{2k}{k}\binom{2n-2k}{n-k}=4^{n},
\]
that can be deduced from the Taylor series
\[
\sum_{n\ge0}\binom{2n}{n}z^{n}=\frac{1}{\sqrt{1-4z}}.
\]
To make this analogue clearer, let us first rewrite this identity in terms of the $\star$ operator \eqref{star} as
\[
\binom{2n}{n} \star \binom{2n}{n}=4^{n} - 2 \binom{2n}{n}.
\]

\begin{thm}
The convolution of continuous central binomial coefficients is given by
\begin{equation}
\left(\left\{ \begin{array}{c}
2s\\
s
\end{array}\right\} +\delta\left(s\right)\right)*\left(\left\{ \begin{array}{c}
2s\\
s
\end{array}\right\} +\delta\left(s\right)\right)=4e^{2s}+\delta\left(s\right)\label{eq:conv central},
\end{equation}
or equivalently by
\[
\left\{ \begin{array}{c}
2s\\
s
\end{array}\right\} *\left\{ \begin{array}{c}
2s\\
s
\end{array}\right\} =4e^{2s}-2\left\{ \begin{array}{c}
2s\\
s
\end{array}\right\} .
\]
This can be generalized to any $2n-$tuple convolution as 
\[
\left(\left\{ \begin{array}{c}
2s\\
s
\end{array}\right\} +\delta\left(s\right)\right)^{*2n}=4e^{2s}L_{n-1}^{\left(1\right)}\left(-4s\right)+\delta\left(s\right)
\]
where $L_{n}^{\left(k\right)}\left(x\right)$ is the  associated Laguerre polynomial
with Rodrigues formula
\[
L_{n}^{\left(k\right)}\left(x\right)=\frac{e^{x}x^{-k}}{n!}\frac{d^{n}}{dx^{n}}\left(e^{-x}x^{n+k}\right).
\]

\end{thm}

\begin{proof}
Expanding
\[
\left(\sqrt{\frac{p+2}{p-2}}\right)^{2n}=\left(1+\frac{4}{p-2}\right)^{n}=\sum_{m=0}^{n}\binom{n}{m}\left(\frac{4}{p-2}\right)^{m}
\]
produces the inverse Laplace transform
\begin{align*}
\sum_{m=1}^{n}\binom{n}{m}\frac{2^{2m-3}e^{2s}s^{m-1}}{3\left(m-1\right)!}+\delta\left(s\right) & =\delta\left(s\right)+e^{2s}\sum_{m=0}^{n-1}\binom{n}{m+1}4.2^{2m}\frac{s^{p}}{p!}\\
 & =\delta\left(s\right)+4e^{2s}L_{n-1}^{\left(1\right)}\left(-4s\right).
\end{align*}
\end{proof}

\section{The continuous binomial distribution}\label{section4}

Following Cano and D\'{i}az \cite{Diaz2}, the continuous binomial coefficients allow to define a continuous version of the discrete binomial distribution through the probability density function
\begin{equation}\label{eq:ContDef}
f_{x,p}\left(s\right)=\frac{1}{A_{x,p}}\cbinom{x}{s} p^{s}\left(1-p\right)^{x-s},\thinspace\thinspace0\le s\le x
\end{equation}
where $0\le p\le1$ and the normalization constant $A_{x,p}$ is such
that
\[
\int_{0}^{x}f_{x,p}\left(s\right) ds=1.
\]

Notice that the centered version of this distribution, namely the distribution of the shifted random variable
\[
Y=X-\frac{x}{2},
\]
where $X$ is distributed as in (\ref{eq:ContDef}), is studied in \cite{Diaz2}. Its density is
\begin{equation}
\frac{1}{A_{x,p}}\cbinom{x}{\frac{x}{2}+s}p^{s+\frac{x}{2}}\left(1-p\right)^{\frac{x}{2}-s},\thinspace\thinspace-\frac{x}{2}\le s\le\frac{x}{2}.\label{eq:centered binomial}
\end{equation}

The normalization constant $A_{x,p}$ of this density is not evaluated in \cite{Diaz2}; we give its value
as follows.
\begin{thm}
\label{normalization}
The normalization constant of the continuous binomial distribution
is equal to 
\begin{align*}
A_{x,p} & =2\left[p\left(1-p\right)\right]^{\frac{x}{2}}\left\{ \cosh\left(\frac{x}{2}\sqrt{4+\log^{2}\frac{p}{1-p}}\right)-\cosh\left(\frac{x}{2}\log\frac{p}{1-p}\right)\right.\\
 & \left.+\frac{2}{\sqrt{4+\log^{2}\frac{p}{1-p}}}\sinh\left(\frac{x}{2}\sqrt{4+\log^{2}\frac{p}{1-p}}\right)\right\}.
\end{align*}
\end{thm}
\begin{proof}
Since	
\[
A_{x,p}=\int_{0}^{x}\cbinom{x}{s} p^{s}\left(1-p\right)^{x-s}ds=\left(1-p\right)^{x}\int_{0}^{x}\cbinom{x}{s} \left(\frac{p}{1-p}\right)^{s}ds,
\]
using (\ref{eq:alpha^se^us}) with $\alpha=\frac{p}{1-p}$ and $u=0$
yields the result.
\end{proof}

Moreover, the moment generating function of the continuous binomial distribution \eqref{eq:ContDef} can  be computed explicitly as
follows.
\begin{thm}\label{MomentGen}
The moment generating function of the continuous binomial distribution \eqref{eq:ContDef} is
\begin{equation}
\mathbb{E}e^{uX}=e^{\frac{ux}{2}}\frac{\varphi_{x,p}\left(u\right)}{\varphi_{x,p}\left(0\right)}\label{eq:moment generating X2},
\end{equation}
with
\begin{align*}
\varphi_{x,p}\left(u\right) & =\cosh\left(\frac{x}{2}\sqrt{4+\left(u+\log\frac{p}{1-p}\right)^{2}}\right)-\cosh\left(\frac{x}{2}\log\frac{p}{1-p}\right)\\
 & +\frac{2}{\sqrt{4+\left(u+\log\frac{p}{1-p}\right)^{2}}}\sinh\left(\frac{x}{2}\sqrt{4+\left(u+\log\frac{p}{1-p}\right)^{2}}\right).
\end{align*}
\end{thm}
We remark that $\varphi_{x,p}\left(z\right)=\varphi_{x,\frac{1}{2}}\left(z+\log\frac{p}{1-p}\right)$.


In the symmetric case $p=\frac{1}{2},$ the moments can be explicitly
computed, following  the approach used by S.M. Iacus and N. Yoshida  in the case of the telegraph process \cite{Iacus}.

\begin{thm}
The moments of a random variable $X$ distributed according to the symmetric discrete binomial distribution \eqref{eq:centered binomial} density with $p=\frac{1}{2}$ are
\[
\mathbb{E}X^k=
\begin{cases}
\frac{1}{\left(e^{x}-1\right)}\left[\left(\frac{x}{2}\right)^{\frac{k+1}{2}}\Gamma\left(\frac{k+1}{2}\right)\left(I_{\frac{k+1}{2}}\left(x\right)+I_{\frac{k-1}{2}}\left(x\right)\right)-\left(\frac{x}{2}\right)^{k}\right], & k \text{ even,}\\
0, & k \text{ odd.}
\end{cases}
\]
\end{thm}

\begin{proof}
The density in the case $p=\frac{1}{2}$ is
\[
f_{x}\left(s\right)=\frac{1}{2\left(e^{x}-1\right)}\left\{ \begin{array}{c}
x\\
\frac{x}{2}+s
\end{array}\right\} ,
\]
so that
\begin{align*}
\mathbb{E}X^{k} & =\int_{-\frac{x}{2}}^{+\frac{x}{2}}s^{k}f_{x}\left(s\right)ds=\frac{1}{2\left(e^{x}-1\right)}\int_{-\frac{x}{2}}^{+\frac{x}{2}}s^{k}\left[2I_{0}\left(2\sqrt{\frac{x^{2}}{4}-s^{2}}\right)\right]ds\\
 & +\frac{1}{2\left(e^{x}-1\right)}\int_{-\frac{x}{2}}^{+\frac{x}{2}}s^{k}\left[\frac{x}{\sqrt{\frac{x^{2}}{4}-s^{2}}}I_{1}\left(2\sqrt{\frac{x^{2}}{4}-s^{2}}\right)\right]ds.
\end{align*}
Using
\[
2\int_{-\frac{x}{2}}^{+\frac{x}{2}}s^{k}\left[2I_{0}\left(2\sqrt{\frac{x^{2}}{4}-s^{2}}\right)\right]ds=\left(1+\left(-1\right)^{k}\right)\left(\frac{x}{2}\right)^{\frac{k+1}{2}}\Gamma\left(\frac{k+1}{2}\right)I_{\frac{k+1}{2}}\left(x\right)
\]
and
\[
x\int_{-\frac{x}{2}}^{+\frac{x}{2}}s^{k}\left[\frac{1}{\sqrt{\frac{x^{2}}{4}-s^{2}}}I_{1}\left(2\sqrt{\frac{x^{2}}{4}-s^{2}}\right)\right]ds=\left(1+\left(-1\right)^{k}\right)\left\{ \left(\frac{x}{2}\right)^{\frac{k+1}{2}}\Gamma\left(\frac{k+1}{2}\right)I_{\frac{k-1}{2}}-\left(\frac{x}{2}\right)^{k}\right\} ,
\]
we obtain the result.
\end{proof}
The continuous binomial distribution is illustrated for $x=10$
and for the 3 values $p=\frac{1}{2}$ (symmetric curve), $p=\frac{1}{3}$
and $p=\frac{1}{4}.$
\begin{center}
\begin{figure}

\begin{centering}
\includegraphics[angle=0,origin=c,scale=0.25]{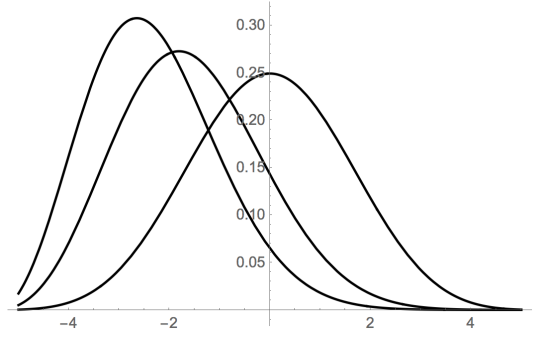}\caption{The continuous binomial distribution for $p=\frac{1}{2},\thinspace\thinspace\frac{1}{3}$,
and $\frac{1}{4}$, right to left.}
\par\end{centering}
\end{figure}
\par\end{center}

\section{A Stochastic Representation}
\label{section5}

The continuous binomial coefficient can be related to a stochastic process, the Goldstein-Kac telegraph process. This was studied by E. Orsingher in \cite{Orsingher} and a complete introduction to this process is given in \cite{Kolesnik}. The Goldstein-Kac process describes successive changes of a binary state,  the number of these changes  following a Poisson distribution: this implies that the successive times spent in each state - in our case, the lengths traveled in each successive direction - are independently and uniformly distributed. This corresponds to the least informative (maximum entropy) among all bounded support distributions.

Consider a Poisson process $n\left(t\right)$ with parameter $\lambda>0,$
and a particle that travels on the real axis, starting from $0$ with
an initial velocity equal to $+c$ or $-c$ each with probability
$1/2.$ The velocity of the particle is supposed to be
\[
v\left(t\right)=\pm c\left(-1\right)^{n\left(t\right)}
\]
so that the particle changes instantaneously the sign of its constant
velocity $c$ at each Poisson event. One trajectory of the velocity
in the case $\lambda=1.3$ and $t\in\left[0,15\right]$ is given below.
\begin{center}
\begin{figure}

\begin{centering}
\includegraphics[scale=0.25]{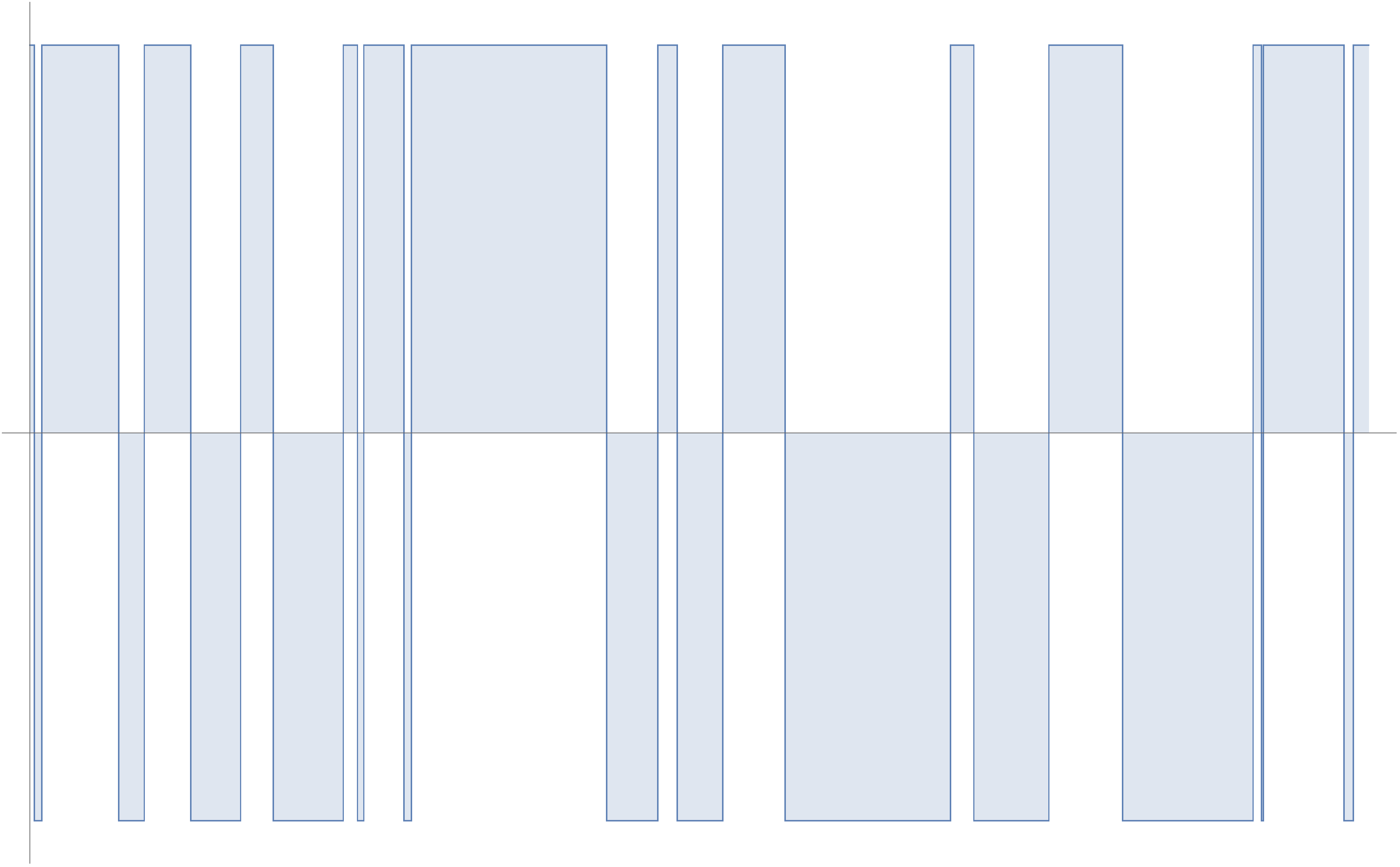}\caption{A trajectory of a telegraph process}
\par\end{centering}
\end{figure}
\par\end{center}

\begin{flushleft}
The location $X\left(t\right)$ of the particle at time $t$, given
by
\begin{equation}
X\left(t\right)=\int_{0}^{t}v\left(\tau\right)d\tau=\pm c\int_{0}^{t}\left(-1\right)^{n\left(\tau\right)}d\tau,\label{eq:X(t)}
\end{equation}
defines the Goldstein-Kacprocess. The probability function of the
location of the particle at time $t$ has two parts:
\par\end{flushleft}

- The discrete part is
\[
\Pr\left\{ X\left(t\right)=ct\vert n\left(t\right)=0\right\} =\Pr\left\{ X\left(t\right)=-ct\vert n\left(t\right)=0\right\} =\frac{1}{2}e^{-\lambda t},
\]
which is the conditional probability that the particle has reached
position $\pm ct$ at time $t$ without any Poisson event happening
since it started at time $0$

- Conditionally to the event $n\left(t\right)>0,$ the probability
function $p\left(s,t\right)ds=\Pr\left\{ s\le X\left(t\right)<s+ds\right\} $
is continuous and its density is given by
\begin{equation}
p\left(s,t\right)=\frac{e^{-\lambda t}}{2c}\left[\lambda I_{0}\left(\frac{\lambda}{c}\sqrt{c^{2}t^{2}-s^{2}}\right)+\frac{\lambda}{c}\frac{tc^{2}}{\sqrt{c^{2}t^{2}-s^{2}}}I_{1}\left(\frac{\lambda}{c}\sqrt{c^{2}t^{2}-s^{2}}\right)\right].\label{eq:p(s,t)}
\end{equation}
Now take
\[
ct=\frac{x}{2}\thinspace\thinspace\text{and}\thinspace\thinspace\lambda=2c=\frac{x}{t}
\]
so that
\begin{align*}
p\left(s,t\right) & =\frac{e^{-x}}{2}\left[2I_{0}\left(2\sqrt{\frac{x^{2}}{4}-s^{2}}\right)+\frac{x}{\sqrt{\frac{x^{2}}{4}-s^{2}}}I_{1}\left(2\sqrt{\frac{x^{2}}{4}-s^{2}}\right)\right]\\
 & =\frac{e^{-x}}{2}\cbinom{x}{\frac{x}{2}+s}.
\end{align*}
In the discrete setup of a centered binomial distribution with $p=1-p=\frac{1}{2},$
the usual binomial coefficient $\binom{n}{k}$ is proportional to
the numbers of ways that the particle, starting from $0,$ can reach
the site $k$ after $n$ independent equiprobable jumps to the left
or to the right. Assuming $n$ even, we have $-\frac{n}{2}\le k\le\frac{n}{2}.$

Similarly, the continuous binomial coefficient measures the ``number''
of continuous paths of an integrated telegraphic random process that,
starting from $\left(0,0\right),$ reach the point $\left(\frac{x}{2}+s,\frac{x}{2}-s\right)$,
traveling horizontally during a total time $\frac{1}{c}\left(\frac{x}{2}+s\right)$
and vertically during a remaining total time $\frac{1}{c}\left(\frac{x}{2}-s\right),$
and 
switching between East and North-East directions each time a Poisson event happens. The attached
figure shows two trajectories of such a process.
\begin{center}
\begin{figure}
\label{trajectories}
\begin{centering}
\includegraphics[scale=0.25]{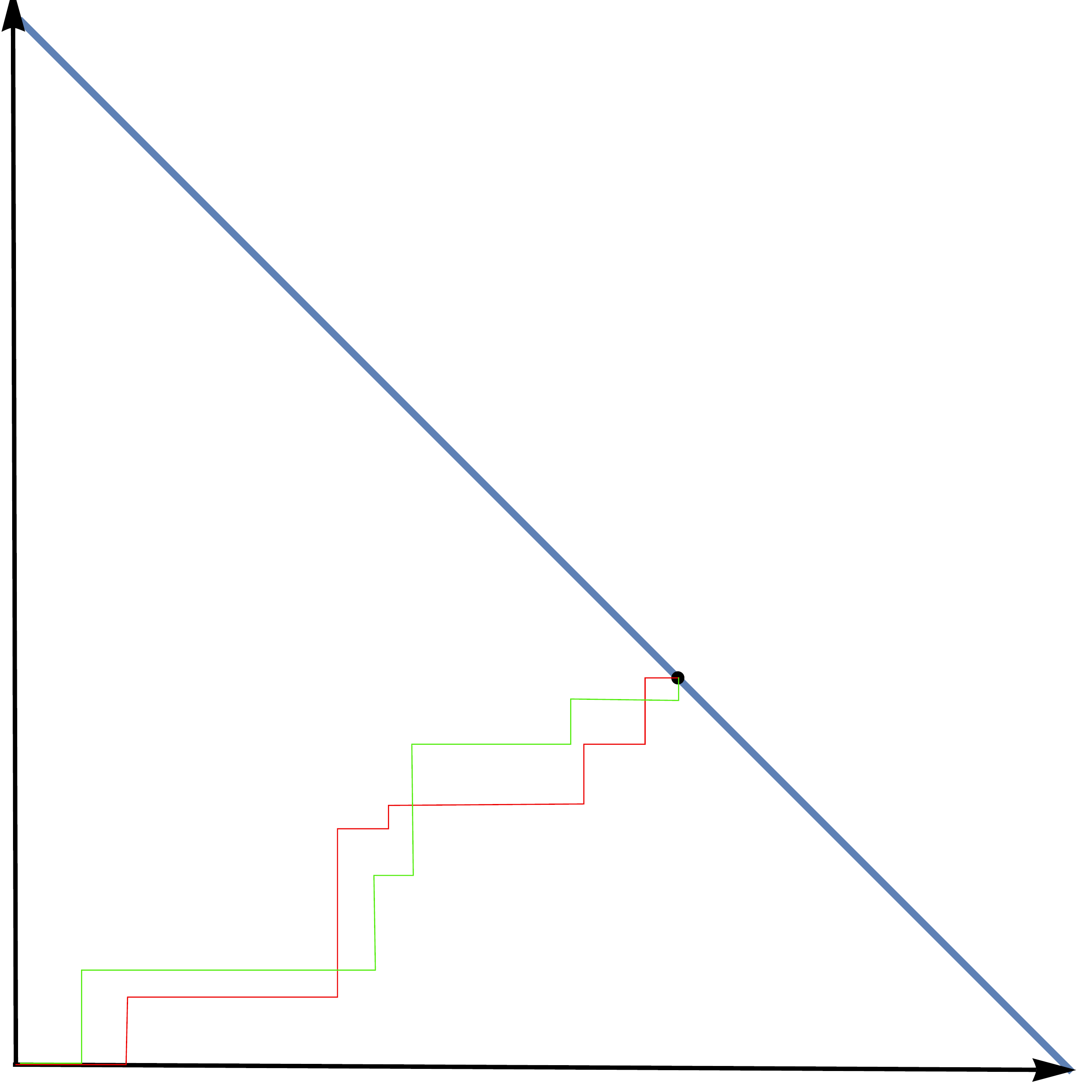}\caption{Two trajectories induced by the integrated telegraph process}
\par\end{centering}
\end{figure}
\par\end{center}

Note that the density (\ref{eq:p(s,t)}) satisfies the differential
equation
\begin{equation}
c^{2}\frac{\partial^{2}p}{\partial x^{2}}=\frac{\partial^{2}p}{\partial t^{2}}+2\lambda\frac{\partial p}{\partial t}\label{eq:diffeq},
\end{equation}
which can be transformed into
\[
c^{2}\frac{\partial^{2}v}{\partial x^{2}}+\lambda v^{2}=\frac{\partial^{2}v}{\partial t^{2}}
\]
with $v\left(x,t\right)e^{-\lambda t}=p\left(x,t\right).$ 

%
%

\section{Next Steps}
In this work, we studied coutinuous analogs of the binomial coefficient and Catalan numbers, and showed that they possess several properties of independent interest. Compact expression for both in terms of Bessel $I$ functions should allow us to prove several straightforward results about them in the future. Because of a reduction procedure to the discrete case, described in \cite{Wakhare}, this can potentially inform research about the discrete case.

\end{document}